\documentclass[11pt,reqno]{amsart}
\usepackage{amssymb,latexsym}
\usepackage{graphicx}
\usepackage{url}		

\calclayout

\usepackage[utf8]{inputenc}
\usepackage[T1]{fontenc}
\usepackage{amsfonts}
\usepackage[greek,english]{babel}
\usepackage{lmodern}

\usepackage{enumerate}
\usepackage{enumitem}

\usepackage[normalem]{ulem} 

\makeatletter
\newcommand{\biggg}{\bBigg@{2.75}}
\newcommand{\Biggg}{\bBigg@{3}}
\newcommand{\bigggg}{\bBigg@{3.25}}
\newcommand{\Bigggg}{\bBigg@{3.5}}
\newcommand{\biggggg}{\bBigg@{3.75}}
\newcommand{\Biggggg}{\bBigg@{4}}
\makeatother

\newtheoremstyle{mystyle}
  {}
  {}
  {\itshape}
  {}
  {\bfseries}
  {.}
  { }
  {\thmname{#1}\thmnumber{ #2}\thmnote{ (#3)}} 

\theoremstyle{mystyle}

\newtheorem{thm}{Theorem}[section]
\newtheorem{theorem}[thm]{Theorem}
\newtheorem{lemma}[thm]{Lemma}

\newtheorem{entry}[thm]{Entry}

\makeatletter
\newcommand{\leqnomode}{\tagsleft@true\let\veqno\@@leqno}
\newcommand{\reqnomode}{\tagsleft@false\let\veqno\@@eqno}
\makeatother
\allowdisplaybreaks
\begin{document}
\setcounter{page}{1}

\title{The three missing terms in \linebreak Ramanujan's septic theta function identity}
\author{Örs Rebák}
\address{Department of Mathematics and Statistics, University of Tromsø -- The Arctic University of Norway, 9037 Tromsø, Norway}
\email{ors.rebak@uit.no}

\begin{abstract}
On page 206 in his lost notebook, Ramanujan recorded the following enigmatic identity for his theta function $\varphi(q)$:
\begin{equation*}
\varphi(e^{-7\pi\sqrt{7}}) = 7^{-3/4}\varphi(e^{-\pi\sqrt{7}})\big\{1 + (\quad)^{2/7} + (\quad)^{2/7} + (\quad)^{2/7}\big\}.
\end{equation*}
We give the three missing terms. In addition, we calculate the class invariant $G_{343}$ and further special values of $\varphi(e^{-n\pi}),$ for $n = 7, 21, 35,$ and $49$.
\end{abstract}

\date{\today}

\maketitle

\section{Introduction}

Ramanujan's general theta function $f(a,b)$ is defined by \cite[p.~197]{RamanujanEarlierII}, \cite[p.~34]{BerndtIII}
\begin{equation}\label{def:generaltheta}
f(a,b) := \sum_{n = -\infty}^\infty a^{n(n+1)/2}b^{n(n-1)/2}, \qquad |ab| < 1,
\end{equation}
which provides an alternative formulation \cite[p.~3]{BerndtClassicalToModern} for the classical theta functions $(\theta_i(z,q))_{i=1}^4$ in \cite[pp.~462--465]{WhittakerWatson}. The symmetry reflected in the definition of $f(a,b)$ is inherited by its representation by the Jacobi triple product identity \cite[pp.~176--183]{Jacobi}, \cite[p.~197]{RamanujanEarlierII}, \cite[p.~35]{BerndtIII}, which states that
\begin{equation}\label{Jacobi-triple-product}
f(a,b) = (-a; ab)_\infty (-b; ab)_\infty (ab; ab)_\infty, \qquad  |ab| < 1,
\end{equation}
where
\begin{align*}
(a; q)_\infty := \prod_{k = 0}^\infty (1 - aq^k), \qquad |q| < 1.
\end{align*}
In Ramanujan's notation, the theta function $\varphi(q)$ is defined by \cite[p.~197]{RamanujanEarlierII}, \cite[p.~36]{BerndtIII}
\begin{equation}\label{def:phi}
\varphi(q) := f(q,q) = \sum_{n = -\infty}^\infty q^{n^2} = (-q; q^2)_\infty^2 (q^2; q^2)_\infty, \qquad |q| < 1,
\end{equation}
where the series and product representations are straightforward from \eqref{def:generaltheta} and \eqref{Jacobi-triple-product}, respectively. Furthermore, we set \cite[p.~197]{RamanujanEarlierII}, \cite[p.~37]{BerndtIII}
\begin{equation}\label{def:chi}
\chi(q) := (-q; q^2)_\infty, \qquad |q| < 1.
\end{equation}

If $q = \exp(-\pi\sqrt{n})$, for a positive rational number $n$, the class invariant $G_n$ of Ramanujan \cite{RamanujanModularPi}, \cite[pp.~23--39]{RamanujanCollected}, \cite[p.~183]{BerndtV} and Weber \cite[p.~114]{Weber} is defined by
\begin{equation}\label{def:G}
G_n := 2^{-1/4}q^{-1/24}\chi(q).
\end{equation}
For odd $n$, Ramanujan's values for $G_n$ are listed in \cite[pp.~189--199]{BerndtV}, Weber's list is in \cite[pp.~721--726]{Weber}, and motivation is in \cite{BerndtChanZhang2}. The class invariants are algebraic \cite[pp.~214, 257]{Cox}. 

A fundamental result in the theory of elliptic functions is that for a positive rational $n$,
\begin{equation}\label{elliptic-fundamental}
\varphi^2(e^{-\pi\sqrt{n}}) = {_2}F{_1}\big(\tfrac12, \tfrac12; 1; k_n^2\big) = \frac{2}{\pi}K(k_n).
\end{equation}
Here, ${_2}F{_1}$ is the ordinary or Gaussian hypergeometric function \cite[pp.~24, 281]{WhittakerWatson}, $K$ is the complete elliptic integral of the first kind \cite[pp.~499--500]{WhittakerWatson}, \cite[p.~102]{BerndtIII}, \cite{BerndtChan}, and $k_n$ is a singular value or singular modulus  \cite[pp.~525--527]{WhittakerWatson}, \cite{JoyceZucker}, \cite{BorweinZucker}, \cite{BerndtChanZhang2}, \cite[p.~183]{BerndtV} of the elliptic integral~$K$. The singular values are algebraic \cite{Abel}. Ramanujan used the notation $\alpha_n := k_n^2$ \cite{BerndtChanZhang2}, \cite[p.~183]{BerndtV}. This statement is given more generally in \cite[p.~207]{RamanujanEarlierII}, \cite[p.~101, Entry~6]{BerndtIII}. An overview of the theory of elliptic functions can be found in \cite[p.~102]{BerndtIII}, \cite{BerndtChan}, \cite{BerndtChanZhang2}, \cite[pp.~323--324]{BerndtV}.

For a positive rational $n$, a positive integer $d$, and $q = \exp(-\pi\sqrt{n})$, in the theory of modular equations, the multiplier $m$ of degree $d$ is defined by
\begin{equation}\label{def:multiplier}
m :=  \frac{\varphi^2(q)}{\varphi^2(q^d)} = \frac{\varphi^2(e^{-\pi\sqrt{n}})}{\varphi^2(e^{-d\pi\sqrt{n}})}.
\end{equation}
The multiplier $m$ can be defined more generally, as given in \cite[p.~230]{BerndtIII}, \cite{BerndtChan}, \cite{BerndtChanZhang2}, \cite[p.~324]{BerndtV}. In our case, $m$ defined in \eqref{def:multiplier} is algebraic \cite{Zucker}. An overview of the theory of modular equations can be found in \cite[pp.~213--214]{BerndtIII}, \cite{BerndtChan}, \cite{BerndtChanZhang2}, \cite[p.~185]{BerndtV}.

It is classical \cite[pp. 524--525]{WhittakerWatson}, and it was also discovered by Ramanujan \cite[p.~207]{RamanujanEarlierII},\linebreak \cite[p.~103]{BerndtIII}, \cite[p.~248]{RamanujanEarlierI}, \cite{BerndtChan}, \cite[p.~325]{BerndtV}, that
\begin{equation}\label{e-pi}
\varphi(e^{-\pi}) = \frac{\pi^{1/4}}{\Gamma\big(\frac{3}{4}\big)} = \frac{\Gamma\big(\frac{1}{4}\big)}{\sqrt{2} \pi^{3/4}}.
\end{equation}
For a positive rational $n$, Ramanujan recorded his values for $\varphi(e^{-n\pi})$ in terms of $\varphi(e^{-\pi})$, but in view of \eqref{e-pi}, $\varphi(e^{-n\pi})$ is therefore determined explicitly. At scattered places in his notebooks, Ramanujan recorded some values of $\varphi(e^{-n\pi})$ when $n$ is a power of two, namely, for $n = 1, 2, 4, 1/2, 1/4$ \cite[p.~248]{RamanujanEarlierI}, \cite[p.~325]{BerndtV}; and when $n \geq 3$ is an odd integer, namely, for $n = 3, 5, 7, 9,$ and $45$ \cite[pp.~284, 285, 297, 287, 312]{RamanujanEarlierI}, \cite{BerndtChan}, \cite[pp.~327--337]{BerndtV}. The values\linebreak at powers of two are parts of more general results from Ramanujan's second notebook\linebreak \cite[p.~210]{RamanujanEarlierII}, \cite[pp.~122--123]{BerndtIII}. The evaluations for the odd values were established by Berndt and Chan \cite{BerndtChan}, \cite[pp.~327--337]{BerndtV}. They also determined the values for $n = 13, 27,$ and $63$.

Selberg and Chowla \cite{SelbergChowla} showed that for any singular value $k_n$, the elliptic integral $K(k_n$) is expressible in terms of gamma functions. J.~M.~Borwein and Zucker \cite{Zucker}, \cite{BorweinZucker}, \cite[p.~298]{BorweinBrothersPiAGM} evaluated $K(k_n)$, for $n = 1, \dots, 16$. Thus, by \eqref{elliptic-fundamental}, we have the value of $\varphi(e^{-\pi\sqrt{n}})$ in these cases. We give two theta function values, corresponding to $k_3$ and $k_7$ \cite{JoyceZucker}, \cite{ZuckerJoyceII}, respectively:
\begin{equation}\label{e-sqrt3}
\varphi(e^{-\pi\sqrt{3}}) = \frac{3^{1/8}\Gamma^{3/2}\big(\frac{1}{3}\big)}{2^{2/3}\pi}
\end{equation}
and
\begin{align}\label{e-sqrt7}
\varphi(e^{-\pi\sqrt{7}}) &= \frac{\big\{\Gamma\big(\frac{1}{7}\big)\Gamma\big(\frac{2}{7}\big)\Gamma\big(\frac{4}{7}\big)\big\}^{1/2}}{\sqrt{2} \cdot 7^{1/8}\pi} = \frac{\sqrt{2}\,\big\{\big(\cos\big(\frac{\pi}{7}\big) - \cos\big(\frac{3\pi}{7}\big)\big)B\big(\frac{1}{7}, \frac{2}{7}\big)\big\}^{1/2}}{7^{3/8}\sqrt{\pi}},
\end{align}
where $B(x,y) \mathop{:=} \Gamma(x)\Gamma(y)/\Gamma(x + y)$, for $\operatorname{Re}(x),\operatorname{Re}(y) \mathop{>} 0$, is the beta function \cite[pp.~253--256]{WhittakerWatson}.

If we would like to calculate $\varphi(e^{-\pi\sqrt{r}})$ explicitly, for a positive integer $r$, then if $r$ is square-free and the corresponding values $k_r$ and $K(k_r)$ are known, we can use \eqref{elliptic-fundamental}. If $r$ is not square-free and the value $\varphi(e^{-\pi\sqrt{n}})$ is known, where $n$ is the square-free part of $r$ and $\sqrt{r} = d\sqrt{n}$, for a positive integer $d$, then we can calculate $\varphi(e^{-d\pi\sqrt{n}})$, with appropriate modular equations of degree $d$, which contains the class invariants $G_{n}$ and $G_{d^2n}$, with known explicit values, and the multiplier $m$. There are other particular methods, as we see next, in Entry~\ref{entry:Ramanujan}.

On page 206 in his lost notebook \cite{RamanujanLost}, Ramanujan recorded the following identities.
\begin{entry}[p. 206]\label{entry:Ramanujan}
\leqnomode
Let
\begin{equation}\label{1+u+v+w}
\frac{\varphi(q^{1/7})}{\varphi(q^7)} = 1 + u + v + w.\tag{i}
\end{equation}
Then
\begin{equation}\label{def:p}
p := uvw = \frac{8q^2 (-q; q^2)_\infty}{(-q^7; q^{14})^7_\infty}\tag{ii}
\end{equation}
and
\begin{align}\label{phi-to-p}
\frac{\varphi^8(q)}{\varphi^8(q^7)} - (2 + 5p)\frac{\varphi^4(q)}{\varphi^4(q^7)} + (1 - p)^3 = 0.\tag{iii}
\end{align}
Furthermore,
\begin{equation}\label{u,v,w}
u = \bigg(\frac{\alpha^2 p}{\beta}\bigg)^{1/7}, \quad
v = \bigg(\frac{\beta^2 p}{\gamma}\bigg)^{1/7}, \text{\quad and \quad}
w = \bigg(\frac{\gamma^2 p}{\alpha}\bigg)^{1/7},\tag{iv}
\end{equation}
where $\alpha, \beta,$ and $\gamma$ are the roots of the equation
\begin{equation}\label{r}
r(\xi) := \xi^3 + 2\xi^2\bigg(1 + 3p - \frac{\varphi^4(q)}{\varphi^4(q^7)}\bigg) + \xi p^2(p+4) - p^4 = 0.\tag{v}
\end{equation}
For example,
\begin{equation}\label{enigmatic}
\varphi(e^{-7\pi\sqrt{7}}) = 7^{-3/4}\varphi(e^{-\pi\sqrt{7}})\Big\{1 + (\quad)^{2/7} + (\quad)^{2/7} + (\quad)^{2/7}\Big\}.\tag{vi}
\end{equation}
\reqnomode
\end{entry}

We remark that \eqref{1+u+v+w}--\eqref{r} hold for $|q| < 1$, with $q \neq 0$ in \eqref{u,v,w}. If $q = 0$, then $u = v = w = 0$. Parts \eqref{1+u+v+w}--\eqref{r} were proved by Son \cite{Son}, \cite[pp.~180--194]{AndrewsBerndtII}, \cite[pp.~198--200]{Son2}.

Part \eqref{1+u+v+w} is recorded in Ramanujan's second notebook \cite[p.~239]{RamanujanEarlierII}, \cite[p.~303]{BerndtIII} as well, in the form of
\begin{equation*}
\varphi(q^{1/7}) - \varphi(q^7) = 2q^{1/7}f(q^5, q^9) + 2q^{4/7}f(q^3, q^{11}) + 2q^{9/7}f(q, q^{13}),
\end{equation*}
from where the values of $u, v,$ and $w$ can be determined \cite{Son}, \cite[p.~181]{AndrewsBerndtII}, \cite[p.~198]{Son2} as
\begin{equation}\label{def:u,v,w}
u:=2q^{1/7}\frac{f(q^5, q^9)}{\varphi(q^7)}, \qquad
v:=2q^{4/7}\frac{f(q^3, q^{11})}{\varphi(q^7)}, \qquad
w:=2q^{9/7}\frac{f(q, q^{13})}{\varphi(q^7)},
\end{equation}
since they are not clearly defined in \eqref{u,v,w}. For \eqref{1+u+v+w} to hold, we could give $u, v,$ and $w$ in any arbitrary order, but throughout the paper, we use the definitions in \eqref{def:u,v,w}.

In \eqref{enigmatic}, Ramanujan gave an enigmatic identity, as a fragmentary example, where on the right-hand side there are three missing terms. Note that Ramanujan used the exponent $2/7$, instead of $1/7$, as he should have according to \eqref{u,v,w}. It turns out that this is correct, so it is likely that Ramanujan knew something about the structure of the terms. Ramanujan wrote $7^{3/4}$ instead of $7^{-3/4}$ on the right-hand side. We have corrected this.

Berndt \cite{BerndtForty}, Son \cite{Son}, and Andrews--Berndt \cite[p.~181]{AndrewsBerndtII} leave the problem of the three missing terms open. They wonder why Ramanujan did not record the terms in \eqref{enigmatic}. We cannot answer this question, but Ramanujan gave us the procedure in \eqref{1+u+v+w}--\eqref{r} that helps us to find them. The equations in Entry~\ref{entry:Ramanujan}\eqref{1+u+v+w}--\eqref{r} can be interpreted as the following:
\begin{enumerate}
\item[\eqref{1+u+v+w}] Our aim is to find the values of $u, v,$ and $w$ for a given $|q| < 1$.
\item[\eqref{def:p}] Calculate $p$.
\item[\eqref{phi-to-p}] Solve the quadratic equation for $\varphi^4(q)/\varphi^4(q^7)$ and choose the \emph{correct root}.
\item[\eqref{r}] By solving the cubic equation $r(\xi) = 0$, find $\alpha, \beta,$ and $\gamma$.
\item[\eqref{u,v,w}] By using $\alpha, \beta,$ and $\gamma$ in the \emph{correct order}, construct $u, v,$ and $w$.
\end{enumerate}

Before we take these steps, we give some preliminaries in Sections~\ref{section:preliminaries} and \ref{section:order}. Ramanujan gave us no hints on which is the correct choice for $\varphi^4(q)/\varphi^4(q^7)$ in \eqref{phi-to-p}, and how to find the correct order of the roots of $r$ in \eqref{r}. Possibly, he stopped after solving \eqref{r}, since the exponents $2/7$ on the right-hand side of \eqref{enigmatic} become apparent after one finds a proper representation for the roots of $r$, but before their correct order is determined. This part needs most of our preparation; thus we give lemmas on the correct order of $\alpha, \beta,$ and $\gamma$ in Section~\ref{section:order}. Our main result is in Section~\ref{section:three-missing-terms}, where we give the three missing terms of \eqref{enigmatic}. In Section~\ref{section:G343}, we give a closed-form expression for the class invariant $G_{343}$.  In Section~\ref{section:examples-1}, we calculate the special values of $\varphi(e^{-n\pi}),$ for $n = 7, 21,$ and $35$, and the value of $\varphi(e^{-7\pi\sqrt{3}})$. In Section~\ref{section:examples-2}, we conclude our article with the value of $\varphi(e^{-49\pi})$, given as a second example of Ramanujan's type for \eqref{1+u+v+w}.

\section{Preliminaries}\label{section:preliminaries}

We recall the transformation formula for $\varphi(e^{-\pi\sqrt{n}})$.

\begin{lemma}\label{lemma:transform} If $n$ is a positive rational number, then
\begin{equation*}
\varphi(e^{-\pi/\sqrt{n}}) = n^{1/4} \varphi(e^{-\pi\sqrt{n}}).
\end{equation*}
\end{lemma}

\begin{proof} The transformation formula for $\varphi(q)$ states that \cite[p.~199]{RamanujanEarlierII}, \cite[p.~43]{BerndtIII} if $a, b > 0$ with $ab = \pi$, then
\begin{equation*}
\sqrt{a} \varphi\big(e^{-a^2}\big) = \sqrt{b} \varphi\big(e^{-b^2}\big).
\end{equation*}
The lemma is the special case for $a^2 = \pi/\sqrt{n}$.
\end{proof}

Ramanujan gave some properties of $G_n$. We need the following.

\begin{lemma}\label{lemma:G} If $n$ is a positive rational number, then $G_n = G_{1/n}$.
\end{lemma}

\begin{proof}
See Ramanujan's paper \cite{RamanujanModularPi}, \cite[pp.~23--39]{RamanujanCollected} or Yi's thesis \cite[pp.~18--19]{YiThesis}.
\end{proof}

Our next lemma helps us to find $p$ in Entry~\ref{entry:Ramanujan}\eqref{def:p}. It gives a connection between $p$~and~$G_n$.

\begin{lemma}\label{lemma:p} If $q = \exp(-\pi\sqrt{n})$, for a positive rational number $n$, then
\begin{equation*}
p = \frac{2\sqrt{2} G_n}{G_{49n}^7}.
\end{equation*}
\end{lemma}

\begin{proof}
From Entry~\ref{entry:Ramanujan}\eqref{def:p}, \eqref{def:chi}, and \eqref{def:G}, we have
\begin{equation*}
p = \frac{8q^2 (-q; q^2)_\infty}{(-q^7; q^{14})^7_\infty} = \frac{8q^2\chi(q)}{\chi^7(q^7)} = \frac{2\sqrt{2} G_n}{G_{49n}^7}.\tag*{\qedhere}
\end{equation*}
\end{proof}

The Chebyshev polynomial $U_n$ of the second kind \cite[pp.~3--4]{MasonHandscomb} is defined for $\left|\cos\theta\right| \leq 1$ by
\begin{equation}\label{def:U}
U_n(\cos\theta) := \frac{\sin((n+1)\theta)}{\sin\theta}, \qquad n = 0, 1, \dots.
\end{equation}
The polynomial $U_n$ satisfies the recurrence relation $U_0(x) = 1, U_1(x) = 2x$,
\begin{equation*}
U_n(x) = 2xU_{n-1}(x) - U_{n-2}(x), \qquad n = 2, 3, \dots,
\end{equation*}
which extends the definition to all complex values $x$. Thus, $U_n$ is a polynomial of degree $n$, with real, distinct roots, which are symmetric about zero. The first few cases are listed in \cite[p.~994]{GradshteynRyzhik7}.

\begin{lemma}\label{lemma:roots} The roots of $U_n$ are
\begin{equation*}
x_k = \cos\frac{k\pi}{n+1}, \qquad k = 1, \dots, n.
\end{equation*}
\end{lemma}
\begin{proof} This follows directly from the definition \eqref{def:U}.
\end{proof}

\section{The order of the roots}\label{section:order}

In Entry~31 of Chapter~16 of Ramanujan's second notebook \cite[p.~200]{RamanujanEarlierII}, \cite[pp.~48--49]{BerndtIII}, the following general theorem is stated. Let $\mathcal{U}_k = a^{k(k+1)/2}b^{k(k-1)/2}$ and $\mathcal{V}_k = a^{k(k-1)/2}b^{k(k+1)/2}$ for each nonnegative integer $k$. Then, for a positive integer $n$,
\begin{equation}\label{entry31}
f(\mathcal{U}_1, \mathcal{V}_1) =  \sum_{k=0}^{n-1} \mathcal{U}_k f\bigg(\frac{\mathcal{U}_{n+k}}{\mathcal{U}_k},\frac{\mathcal{V}_{n-k}}{\mathcal{U}_k}\bigg), \qquad |ab| < 1, \quad ab \neq 0.
\end{equation}
For a positive integer $n$, and $|q| < 1$, $q \neq 0$, let
\begin{align}\label{def:u_k}
u_k := q^{k^2/n}f(q^{n + 2k}, q^{n - 2k}) = \sum_{m=-\infty}^{\infty} q^{(k-mn)^2/n}, \qquad k = 0, \dots, n - 1,
\end{align}
where the series representation is obtained by \eqref{def:generaltheta}. Note by \eqref{def:phi} that $u_0 = \varphi(q^n)$. By setting $(a, b) = (q^{1/n}, q^{1/n})$ into \eqref{entry31}, we obtain
\begin{equation*}
\varphi(q^{1/n}) = \sum_{k=0}^{n-1} u_k.
\end{equation*}
We consider $f(a,b) = af(a^{-1}, a^2b)$, for $|ab| < 1$ and $a \neq 0$, which is followed by Entry~18(i),(iv) in \cite[p.~197]{RamanujanEarlierII}, \cite[pp.~34--35]{BerndtIII}. Because of this, we have $u_k = u_{n-k},$ for $k = 1,\dots,n-1.$ Since $u_0$ is nonzero, for each odd integer $n$,
\begin{equation*}
\frac{\varphi(q^{1/n})}{u_0} = 1 + \sum_{k=1}^{(n-1)/2} \frac{2u_k}{u_0}.
\end{equation*}
For $n = 7$, we arrive at Entry~\ref{entry:Ramanujan}\eqref{1+u+v+w}, and for $u, v,$ and $w,$ defined in \eqref{def:u,v,w}, we have
\begin{equation}\label{u,v,w-u_k}
u = \frac{2u_1}{u_0}, \qquad v = \frac{2u_2}{u_0}, \text{\quad and \quad}w = \frac{2u_3}{u_0}.
\end{equation}

For finding the three missing terms, it is enough to handle $u, v,$ and $w$ for $0 < q < 1$, in which case these values are real. Thus in this section, we state our lemmas under this condition. 

First, we would like to show that for $0 < q < 1$, the values $u_k$ defined in \eqref{def:u_k} are in descending order, for $k = 0,\dots,\lfloor n/2\rfloor$, where $\lfloor n/2 \rfloor$ is the largest integer $r$, such that $r \leq n/2$. For this purpose, we now overview some classical results.

The third Jacobi theta function $\theta_3(z, q)$ is defined by \cite[pp.~463--464]{WhittakerWatson}
\begin{equation}\label{def:theta_3-q}
\theta_3(z, q) := \sum_{n=-\infty}^{\infty} q^{n^2}e^{2niz}, \qquad z \in \mathbb{C}, \quad |q| < 1,
\end{equation}
or equivalently, $\theta_3(z \mid \tau)$ is defined by
\begin{equation}\label{def:theta_3-tau}
\theta_3(z \mid \tau) := \sum_{n=-\infty}^{\infty} e^{\pi i \tau n^2 + 2niz}, \qquad z \in \mathbb{C}, \quad \operatorname{Im} \tau > 0.
\end{equation}
With $q = e^{\pi i \tau}$, \eqref{def:theta_3-q} and \eqref{def:theta_3-tau} are equal, and $|q| < 1$ if and only if $\operatorname{Im} \tau > 0$. We use both notations, depending on whether we would like to indicate the dependence on $q$ or $\tau$. Using the identity $e^{2niz} + e^{-2niz} = 2\cos(2nz)$, we can rewrite \eqref{def:theta_3-q} as \cite[pp.~463--464]{WhittakerWatson}
\begin{equation}\label{eq:theta_3-trig}
\theta_3(z, q) = 1 + 2\sum_{n=1}^{\infty} q^{n^2}\cos 2nz, \quad z \in \mathbb{C}, \quad |q| < 1.
\end{equation}
A straightforward calculation shows the connection between the third Jacobi theta function and Ramanujan's general theta function. From \eqref{def:theta_3-q} and \eqref{def:generaltheta}, we find that \cite[p.~3]{BerndtIII}
\begin{equation}\label{Ramanujan-Jacobi-connection-q}
\theta_3(z, q) = f(qe^{2iz}, qe^{-2iz}), \qquad z \in \mathbb{C}, \quad |q| < 1, 
\end{equation}
and from \eqref{def:theta_3-tau} and \eqref{def:generaltheta}, we have
\begin{equation}\label{Ramanujan-Jacobi-connection-tau}
\theta_3(z \mid \tau) = f(e^{\pi i \tau + 2iz}, e^{\pi i \tau - 2iz}), \qquad z \in \mathbb{C}, \quad \operatorname{Im} \tau > 0. 
\end{equation}

In the following lemma, we show a monotonicity property of the third Jacobi theta function.

\begin{lemma}\label{lemma:theta_3-monotonicity}
If $m$ is an integer and $0 < q < 1$, then
\begin{enumerate}[label=\emph{(\roman*)}, ref=(\roman*)]
\item if $z \in [m\pi, m\pi + (\pi/2)]$, then $\theta_3(z,q)$ is strictly monotonically decreasing in $z$,\label{theta_3-monotonicity-dec} and
\item if $z \in [m\pi - (\pi/2), m\pi]$, then $\theta_3(z,q)$ is strictly monotonically increasing in $z$.\label{theta_3-monotonicity-inc}
\end{enumerate}
\end{lemma}

\begin{proof} Since $\theta_3(z, q)$ has period $\pi$ in $z$ \cite[p.~463]{WhittakerWatson}, it is enough to show that the statement is true for $m = 0$. Because $\theta_3(z, q)$ is an even function of $z$ \cite[p.~464]{WhittakerWatson}, it is enough to show one of the two cases.

To prove \ref{theta_3-monotonicity-dec}, we consider $\theta_3(z,q)$ for $z \in [0, \pi/2]$ and $0 < q < 1$. The zeros of $\theta_3(z,q)$ are of the form of $z = (k + (1/2))\pi + (\ell + (1/2))\pi\tau$, where $\operatorname{Im} \tau > 0$, for all integer values of $k$ and $\ell$ \cite[pp.~465--466]{WhittakerWatson}, therefore it has no zeros for $z \in [0, \pi/2]$. Since the series for $\theta_3(z,q)$ in  \eqref{def:theta_3-q} is a series of analytic functions, uniformly convergent in any bounded domain of values of $z$ \cite[p.~463]{WhittakerWatson}, $\theta_3(z, q)$ is therefore a continuous function. Furthermore, from \eqref{def:theta_3-q} or \eqref{eq:theta_3-trig}, $\theta_3(z,q)$ is a real-valued function, for that $\theta_3(0, q) = \sum_{n = -\infty}^{\infty} q^{n^2} > 0$, thus by the contrapositive of the intermediate value theorem, we find that $\theta_3(z, q)$ is positive for $z \in [0, \pi/2]$.

From \eqref{Ramanujan-Jacobi-connection-q} and from the Jacobi triple product identity \eqref{Jacobi-triple-product}, we find that \cite[p.~469]{WhittakerWatson}
\begin{align}
\theta_3(z,q) &= f(qe^{2iz}, qe^{-2iz}) = (-qe^{2iz}; q^2)_\infty (-qe^{-2iz}; q^2)_\infty (q^2; q^2)_\infty\notag \\
&= \prod_{n=1}^{\infty} (1 - q^{2n})(1 + 2q^{2n - 1}\cos2z + q^{4n-2}).\label{theta_3-Jacobi-triple-product}
\end{align}
Since the resulting series converges uniformly \cite[pp.~471, 79]{WhittakerWatson}, we may differentiate the logarithm of \eqref{theta_3-Jacobi-triple-product} with respect to $z$. Denoting the first partial derivative of $\theta_3(z, q)$ with respect to $z$ by $\theta'_3(z, q)$, and taking the logarithmic derivative of \eqref{theta_3-Jacobi-triple-product}, we find that \cite[p.~489]{WhittakerWatson}
\begin{equation}\label{theta_3-log}
\frac{\theta'_3(z,q)}{\theta_3(z,q)} = -4 \sum_{n=1}^{\infty} \frac{q^{2n - 1}\sin 2z}{1 + 2q^{2n - 1}\cos 2z + q^{4n - 2}}.
\end{equation}
Now, note that the denominator of the summand on the right-hand side of \eqref{theta_3-log} is positive, since
\begin{equation*}
-\frac{1 + q^{4n - 2}}{2q^{2n - 1}} = -\frac{1}{2}\bigg(q^{2n - 1} + \frac{1}{q^{2n - 1}}\bigg) < -1 \leq \cos 2z, \qquad n = 1,2,\dots.
\end{equation*}
Thus, the sign of the sum is depending only on the sign of $\sin 2z$. Since $\theta_3(z,q) > 0$, and since $\sin 2z = 0$ for $z \in \{0, \pi/2\}$ and $\sin 2z > 0$ for $z \in (0, \pi/2)$, we find that $\theta'_3(z,q) = 0$ for $z \in \{0, \pi/2\}$ and $\theta'_3(z,q) < 0$ for $z \in (0,\pi/2)$. Since $\theta_3(z, q)$ is continuous for $z \in [0, \pi/2]$, we conclude that $\theta_3(z,q)$ is strictly monotonically decreasing for $z \in [0, \pi/2]$.
\end{proof}

Now, we prove the needed monotonicity property of the values $u_k$ defined in \eqref{def:u_k}.

\begin{lemma}\label{lemma:order} If $n$ is a nonnegative integer and $0 < q < 1$, then $u_k$ is positive and strictly monotonically decreasing for $k = 0, \dots, n/2$, when $n$ is even, and for $k = 0, \dots, (n-1)/2$, when $n$ is odd.
\end{lemma}
\begin{proof}
Let $n$ and $q$ be fixed with the given conditions. From \eqref{def:u_k}, we deduce that $u_k$ is positive.

To prove the monotonicity, we rewrite $u_k$ in terms of the third Jacobi theta function. From \eqref{def:u_k}, \eqref{def:generaltheta}, \eqref{def:theta_3-tau}, and \eqref{Ramanujan-Jacobi-connection-tau}, we find that
\begin{equation}\label{u_k-theta_3}
u_k = q^{k^2/n}f(q^{n+2k}, q^{n-2k}) = q^{k^2/n}f(e^{\pi i \tau + 2iz_k}, e^{\pi i \tau - 2iz_k}) = q^{k^2/n}\theta_3(z_k \mid \tau),
\end{equation}
where here and in the rest of the proof $k = 0, \dots, \lfloor n/2 \rfloor$,
\begin{equation}\label{z_k-and-tau}
z_k = -i k \log q,  \quad \text{and} \quad \tau = -\frac{i n \log q}{\pi} = iC^{-1}, \quad \text{with} \quad C := \frac{\pi}{n\left|\log q\right|} > 0.
\end{equation}
Since $\log q < 0$, we have $\operatorname{Im} \tau > 0$. Note that $\tau$ is independent of $k$, and that $e^{\pi i \tau} = e^{-\pi / C} = q^n$.

Now, we apply Jacobi's imaginary transformation formula \cite[pp.~474--476]{WhittakerWatson}, \cite[pp.~140--141]{BerndtIV}
\begin{equation}\label{Jacobi-imaginary}
\theta_3(z_k \mid \tau) = (-i\tau)^{-1/2}\exp\bigg(\frac{z_k^2}{\pi i \tau}\bigg)\theta_3(z_k/\tau \mid -1/\tau),
\end{equation}
where by \eqref{z_k-and-tau},
\begin{align}
(-i\tau)^{-1/2} &= \sqrt{C}\label{const-itau}
\intertext{and}
\exp\bigg(\frac{z_k^2}{\pi i \tau}\bigg) &= q^{-k^2/n}\label{exp-z_k}.
\end{align}
From \eqref{u_k-theta_3}--\eqref{exp-z_k}, we find that
\begin{equation*}
u_k = q^{k^2/n}\theta_3(z_k \mid \tau) = \sqrt{C}\cdot\theta_3(z_k/\tau \mid -1/\tau) = \sqrt{C}\cdot\theta_3(z_k/\tau, e^{-\pi i / \tau}),
\end{equation*}
where $C$ is some positive value independent of $k$,
\begin{align*}
\frac{z_k}{\tau} =\frac{k\pi}{n}, \quad \text{and} \quad -\frac{1}{\tau} = iC.
\end{align*}
Since $e^{-\pi i / \tau} = e^{-\pi C} \in (0,1)$ and $(z_k/\tau) = (k\pi/n)$ is a strictly monotonically increasing sequence in $[0,\pi/2]$, applying Lemma~\ref{lemma:theta_3-monotonicity}\ref{theta_3-monotonicity-dec} with $m = 0$, we complete the proof.
\end{proof}

The following lemma is a corollary of Lemma~\ref{lemma:order} for $u, v,$ and $w$ defined in \eqref{def:u,v,w}.

\begin{lemma}\label{lemma:u,v,w,p-supp} For $0 < q < 1$,
\begin{enumerate}[label=\emph{(\roman*)}, ref=(\roman*)]
\item\label{u,v,w-supp} $2 > u > v > w > 0$,
\item\label{p-supp} $0 < p < 8$.
\end{enumerate}
\end{lemma}
\begin{proof}
To prove \ref{u,v,w-supp}, we represent $u, v,$ and $w$ as in \eqref{u,v,w-u_k}, and we use Lemma~\ref{lemma:order} with $n = 7$. Part~\ref{p-supp} follows from \ref{u,v,w-supp} with $p = uvw$, as it is defined in Entry~\ref{entry:Ramanujan}\eqref{def:p}.
\end{proof}

We need the following two statements on the values of $u, v,$ and $w$ defined in \eqref{def:u,v,w}. 

\begin{lemma}\label{lemma:Son-supp} For $|q| < 1$,
\leqnomode
\begin{equation}\label{u^3v}
u^3v + v^3w + w^3u = 2\bigg(\frac{\varphi^4(q)}{\varphi^4(q^7)} - 3p - 1\bigg)\tag{i}
\end{equation}
and
\begin{equation}\label{u^7}
u^7 + v^7 + w^7 = \frac{\varphi^8(q)}{\varphi^8(q^7)} - 7(p - 2)\frac{\varphi^4(q)}{\varphi^4(q^7)} + 7p^2 -49p - 15\tag{ii}.
\end{equation}
\reqnomode
\end{lemma}
\begin{proof}
See Son's article \cite{Son} or the Andrews--Berndt book \cite[pp.~185--186, 194]{AndrewsBerndtII}.
\end{proof}

\pagebreak

Next, we give the correct root of Entry~\ref{entry:Ramanujan}\eqref{phi-to-p} for $\varphi^4(q)/\varphi^4(q^7)$, when $0 < q < 1$.

\begin{lemma}\label{lemma:phi-to-p} For $0 < q < 1$,
\begin{equation*}
\frac{\varphi^4(q)}{\varphi^4(q^7)} = 1 + \frac{5p}{2} + \frac12 \sqrt{(2 + 5p)^2 - 4(1 - p)^3}.
\end{equation*}
\end{lemma}
\begin{proof}
From Lemma~\ref{lemma:u,v,w,p-supp}\ref{u,v,w-supp}, we know that $u, v, w > 0$, thus by Lemma~\ref{lemma:Son-supp}\eqref{u^3v} we have
\begin{equation*}
\frac{\varphi^4(q)}{\varphi^4(q^7)} \geq 1 + 3p.
\end{equation*}
By solving Entry~\ref{entry:Ramanujan}\eqref{phi-to-p} for $\varphi^4(q)/\varphi^4(q^7)$, we find that only the given solution satisfies this.
\end{proof}

The cubic polynomial $r$ defined in Entry~\ref{entry:Ramanujan}\eqref{r} has the following property.

\begin{lemma}\label{lemma:distinct,positive-roots} For $0 < q < 1$, $r$ has three distinct positive roots.
\end{lemma}

\begin{proof}
We recall that if a cubic polynomial with real coefficients has a positive discriminant, then it has three distinct real roots. For $|q| < 1$, depending on the value of $\varphi^4(q)/\varphi^4(q^7)$ from Entry~\ref{entry:Ramanujan}\eqref{phi-to-p}, $r$ has one of the following two possible discriminants:
\begin{equation*}
\Delta_{\pm} = p^5\bigg(p^3 + 104p^2 + 608p + 512 \pm (8p^{3/2} + 160\sqrt{p})\sqrt{4p^2 + 13p + 32}\bigg).
\end{equation*}
From Lemma~\ref{lemma:phi-to-p}, we know that for $0 < q < 1$, the discriminant of $r$ is $\Delta_{-}$. It follows by elementary algebra that
\begin{align*}
\Delta_{+}\Delta_{-} = p^{10}(p - 8)^6.
\end{align*}
It is clear that for $p > 0$, we have $\Delta_{+} > 0$, and for $p> 0$ and $p \neq 8$, we have $p^{10}(p - 8)^6 > 0$. Since by Lemma~\ref{lemma:u,v,w,p-supp}\ref{p-supp} if $0 < q < 1$, then $0 < p < 8$; thus we find that $\Delta_{-} > 0$.

From Lemma~\ref{lemma:u,v,w,p-supp} we know that if $0 < q < 1$, then $p > 0$ and $u, v, w > 0$. Thus, from the construction in Entry~\ref{entry:Ramanujan}\eqref{u,v,w} it follows that the roots of $r$ are positive.
\end{proof}

The next lemma helps to choose the correct order of the roots of $r$ in the case of $0 < q < 1$.

\begin{lemma}\label{lemma:a,b,c-order} For $0 < q < 1$, suppose that the roots of $r$ are given in order $(\alpha, \beta, \gamma)$ such that they satisfy the following two conditions:
\leqnomode
\begin{equation}\label{a,b,c-order-cond1}
\frac{\alpha^2 p}{\beta} + \frac{\beta^2 p}{\gamma} + \frac{\gamma^2 p}{\alpha} = \frac{\varphi^8(q)}{\varphi^8(q^7)} - 7(p - 2)\frac{\varphi^4(q)}{\varphi^4(q^7)} + 7p^2 -49p - 15\tag{i}
\end{equation}
and
\begin{equation}\label{a,b,c-order-cond2}
\frac{\alpha^2}{\beta} > \frac{\beta^2}{\gamma} > \frac{\gamma^2}{\alpha}.\tag{ii}
\end{equation}
\reqnomode
Then
\begin{equation*}
u = \bigg(\frac{\alpha^2 p}{\beta}\bigg)^{1/7}, \quad
v = \bigg(\frac{\beta^2 p}{\gamma}\bigg)^{1/7}, \text{\quad and \quad}
w = \bigg(\frac{\gamma^2 p}{\alpha}\bigg)^{1/7}.
\end{equation*}
\end{lemma}

Condition~\eqref{a,b,c-order-cond1} guarantees the correct order of $\alpha, \beta$, and $\gamma$ in Entry~\ref{entry:Ramanujan}\eqref{u,v,w}, so that Entry~\ref{entry:Ramanujan}\eqref{1+u+v+w} holds. Condition~\eqref{a,b,c-order-cond2} provides the correct order of $u, v,$ and $w$, so that \eqref{def:u,v,w} holds.

\begin{proof}
First, for each possible order of $\alpha, \beta,$ and $\gamma$, consider the set of possible values of $u, v,$ and $w$ given in Entry~\ref{entry:Ramanujan}\eqref{u,v,w}. For $(\alpha, \beta, \gamma), (\beta, \gamma, \alpha), (\gamma, \alpha, \beta)$ and for $(\gamma, \beta, \alpha), (\beta, \alpha, \gamma), (\alpha, \gamma, \beta)$, we have
\begin{equation*}
\Bigg\{\bigg(\frac{\alpha^2 p}{\beta}\bigg)^{1/7},\, \bigg(\frac{\beta^2 p}{\gamma}\bigg)^{1/7},\, \bigg(\frac{\gamma^2 p}{\alpha}\bigg)^{1/7}\Bigg\} \text{\quad and \quad} \Bigg\{\bigg(\frac{\gamma^2 p}{\beta}\bigg)^{1/7},\, \bigg(\frac{\beta^2 p}{\alpha}\bigg)^{1/7},\, \bigg(\frac{\alpha^2 p}{\gamma}\bigg)^{1/7}\Bigg\}\,,
\end{equation*}
respectively. Thus, it is enough to consider the order $(\alpha, \beta, \gamma)$ and its reverse $(\gamma, \beta, \alpha)$. By Lemma~\ref{lemma:Son-supp}\eqref{u^7}, we know that for at least one of these two sets it is true that the sum of their seventh powers fulfills the condition in \eqref{a,b,c-order-cond1}. We show that exactly one of the two sets fulfills it. Suppose that
\begin{equation*}
\frac{\alpha^2 p}{\beta} + \frac{\beta^2 p}{\gamma} + \frac{\gamma^2 p}{\alpha} = \frac{\gamma^2 p}{\beta} + \frac{\beta^2 p}{\alpha} + \frac{\alpha^2 p}{\gamma}.
\end{equation*}
After rearrangement, we find that
\begin{equation*}
\frac{p(\alpha - \beta)(\alpha - \gamma)(\beta - \gamma)(\alpha + \beta + \gamma)}{\alpha\beta\gamma} = 0,
\end{equation*}
but this is contradiction, since $p$ is positive by Lemma~\ref{lemma:u,v,w,p-supp}\ref{p-supp} and $\alpha, \beta,$ and $\gamma$ are distinct, positive numbers by Lemma~\ref{lemma:distinct,positive-roots}.

Since $p$ is positive by Lemma~\ref{lemma:u,v,w,p-supp}\ref{p-supp}, the condition \eqref{a,b,c-order-cond2} guarantees that $u > v > w$, which holds by Lemma~\ref{lemma:u,v,w,p-supp}\ref{u,v,w-supp}.
\end{proof}

Lastly, we need the following trigonometric identity.

\begin{lemma}\label{lemma:trig}
We have
\begin{equation*}
\bigg(\frac{\cos\frac{\pi}{7}}{2\cos^2\frac{2\pi}{7}}\bigg)^2 + \bigg(\frac{\cos\frac{2\pi}{7}}{2\cos^2\frac{3\pi}{7}}\bigg)^2 + \bigg(\frac{\cos\frac{3\pi}{7}}{2\cos^2\frac{\pi}{7}}\bigg)^2 = 41.
\end{equation*}
\end{lemma}
\begin{proof}
Let $\theta := \pi / 7$, and let
\begin{align*}
a &:= \cos \theta\phantom{1} = -\cos 6\theta = -\cos 8\theta,\\
b &:= \cos 2\theta = -\cos 5\theta = \cos 12\theta,\\
c &:= \cos 3\theta = -\cos 4\theta = -\cos 10\theta = -\cos 18\theta.
\end{align*}
By using power-reduction \cite[p.~32]{GradshteynRyzhik7} and product-to-sum \cite[p.~29]{GradshteynRyzhik7} identities, we derive that
\begin{align*}\label{lemma:trig-eqn}
&\bigg(\frac{a}{2b^2}\bigg)^2 + \bigg(\frac{b}{2 c^2}\bigg)^2 + \bigg(\frac{c}{2 a^2}\bigg)^2 = \frac{a^6 c^4 + b^6 a^4 +c^6 b^4}{4 (abc)^4}\\
&\qquad = \frac{1}{1024(abc)^4}\big\{(10 + 15b - 6c - a)(3 - 4a + b) + (10 - 15c - 6a + b)(3 + 4b - c)\\ 
&\qquad\qquad\qquad + (10 - 15a + 6b - c)(3 - 4c - a)\big\}\\
&\qquad = \frac{1}{1024(abc)^4}\big\{90 + 116(b - c - a) + 91(ca - ab - bc) + 19(a^2 + b^2 + c^2)\big\}\\
&\qquad = \frac{1}{1024(abc)^4}\Big\{90 + 116(b - c - a) + 91\Big(\tfrac{1}{2}(b - c) - \tfrac{1}{2}(a + c) - \tfrac{1}{2}(a - b)\Big) \\
&\qquad\qquad\qquad + 19\Big(\tfrac{1}{2}(1 + b) + \tfrac{1}{2}(1 - c) + \tfrac{1}{2}(1 - a)\Big)\Big\}\\
&\qquad = \frac{433(b - c - a) + 237}{2048(abc)^4}.
\end{align*}
Since we know \cite{BankoffGarfunkel} that $b - c  - a = -1/2$ and $abc = 1/8$, the proof is complete.
\end{proof}

\section{The three missing terms}\label{section:three-missing-terms}

We complete Ramanujan's enigmatic septic theta function identity Entry~\ref{entry:Ramanujan}\eqref{enigmatic}.

\begin{theorem}\label{thm:enigmatic} We have
\begin{equation*}
\varphi(e^{-7\pi\sqrt{7}}) = 7^{-3/4}\varphi(e^{-\pi\sqrt{7}})\Bigg\{1 + \bigg(\frac{\cos\frac{\pi}{7}}{2\cos^2\frac{2\pi}{7}}\bigg)^{2/7} + \bigg(\frac{\cos\frac{2\pi}{7}}{2\cos^2\frac{3\pi}{7}}\bigg)^{2/7} + \bigg(\frac{\cos\frac{3\pi}{7}}{2\cos^2\frac{\pi}{7}}\bigg)^{2/7}\Bigg\}.
\end{equation*}
\end{theorem}

\begin{proof}
We use the results in Entry~\ref{entry:Ramanujan}\eqref{1+u+v+w}--\eqref{r} with $q = \exp(-\pi/\sqrt{7})$. First, by using Lemma~\ref{lemma:transform}, we rewrite Entry~\ref{entry:Ramanujan}\eqref{1+u+v+w} as
\begin{equation*}
\varphi(e^{-7\pi\sqrt{7}}) = 7^{-3/4}\varphi(e^{-\pi\sqrt{7}})\{1 + u + v + w\}.
\end{equation*}
With $G_7$ given in \cite{Joubert}, \cite{Watson3}, \cite{Watson4}, \cite[p.~189]{BerndtV}, by Lemma~\ref{lemma:G}, we find that $G_{1/7} = G_7 = 2^{1/4}$. By Lemma~\ref{lemma:p} with $n = 1/7$, for Entry~\ref{entry:Ramanujan}\eqref{def:p} we have
\begin{equation*}
p = \frac{2\sqrt{2} G_{1/7}}{G_7^7} = \frac{2\sqrt{2} \cdot 2^{1/4}}{2^{7/4}} = 1.
\end{equation*}
Next, we solve the equation in Entry~\ref{entry:Ramanujan}\eqref{phi-to-p}. By Lemma~\ref{lemma:phi-to-p}, or in this special case by Lemma~\ref{lemma:transform}, we have
\begin{equation*}
\frac{\varphi^4(q)}{\varphi^4(q^7)} = \frac{\varphi^4(e^{-\pi/\sqrt{7}})}{\varphi^4(e^{-\pi\sqrt{7}})} = 7.
\end{equation*}
Now, we have all the coefficients of $r$ given in Entry~\ref{entry:Ramanujan}\eqref{r}. We have to determine the zeros of
\begin{equation*}
r(\xi) = \xi^3 - 6\xi^2 + 5\xi - 1.
\end{equation*}

With an appropriate polynomial transformation, we relate $r$ to $U_6$ defined in \eqref{def:U} or given in \cite[p.~994]{GradshteynRyzhik7}. Note that for $\xi \neq 0$,
\begin{align*}
-(2\xi)^6\,r\bigg(\frac{1}{(2\xi)^2}\bigg) = 64\xi^6 - 80\xi^4 + 24\xi^2 - 1 = U_6(\xi),
\end{align*}
and $r(0) = U_6(0) = -1$. From Lemma~\ref{lemma:roots}, we know that $U_6$ has the roots $\xi_k = \cos(k\pi/7),$ for~$k = 1,\dots, 6$, for which $\xi_k \neq 0$ and $|\xi_k| = |\xi_{7-k}|$. Thus, we find that
\begin{equation*}
r\Bigg(\frac{1}{(2\cos\frac{k\pi}{7})^2}\Bigg) = 0, \qquad k = 1, 2, 3.
\end{equation*}

Lastly, we determine the appropriate order of the roots $\alpha, \beta,$ and $\gamma$ in Entry~\ref{entry:Ramanujan}\eqref{u,v,w}. The choice
\begin{equation*}
(\alpha, \beta, \gamma) = \Bigg(\frac{1}{(2\cos\frac{3\pi}{7})^2}, \frac{1}{(2\cos\frac{2\pi}{7})^2}, \frac{1}{(2\cos\frac{\pi}{7})^2}\Bigg)
\end{equation*}
is correct, since the condition Lemma~\ref{lemma:a,b,c-order}\eqref{a,b,c-order-cond1} holds by Lemma~\ref{lemma:trig}, and Lemma~\ref{lemma:a,b,c-order}\eqref{a,b,c-order-cond2} is satisfied by the inequalities $0 < \cos(3\pi/7) < \cos(2\pi/7) < \cos(\pi/7) < 1$. We arrive at
\begin{equation*}
u = \Bigg(\frac{\cos\frac{2\pi}{7}}{2\cos^2\frac{3\pi}{7}}\Bigg)^{2/7}, \quad 
v = \Bigg(\frac{\cos\frac{\pi}{7}}{2\cos^2\frac{2\pi}{7}}\Bigg)^{2/7}, \text{\quad and \quad}
w = \Bigg(\frac{\cos\frac{3\pi}{7}}{2\cos^2\frac{\pi}{7}}\Bigg)^{2/7},
\end{equation*}
which completes the proof.
\end{proof}

By combining the value \eqref{e-sqrt7} and Theorem~\ref{thm:enigmatic}, we find the evaluation of $\varphi(e^{-7\pi\sqrt{7}})$, i.e.,
\begin{equation*}
\varphi(e^{-7\pi\sqrt{7}}) = \frac{\big\{\Gamma\big(\frac{1}{7}\big)\Gamma\big(\frac{2}{7}\big)\Gamma\big(\frac{4}{7}\big)\big\}^{1/2}}{\sqrt{2} \cdot 7^{7/8}\pi}\Bigg\{1 + \bigg(\frac{\cos\frac{\pi}{7}}{2\cos^2\frac{2\pi}{7}}\bigg)^{2/7} + \bigg(\frac{\cos\frac{2\pi}{7}}{2\cos^2\frac{3\pi}{7}}\bigg)^{2/7} + \bigg(\frac{\cos\frac{3\pi}{7}}{2\cos^2\frac{\pi}{7}}\bigg)^{2/7}\Bigg\}.
\end{equation*}

\section{The class invariant $G_{343}$ in closed-form}\label{section:G343}

Berndt \cite{BerndtForty}, Son \cite{Son}, and Andrews--Berndt \cite[p.~181]{AndrewsBerndtII} proposed the explicit value of the class invariant $G_{343}$ as an open problem. Actually, Watson showed in \cite{Watson3}, \cite{Watson4} that $G_{343} = 2^{1/4}x$, where $x^7 - 7x^6 - 7x^5 - 7x^4 - 1 = 0$. This can be proved by using a modular equation of degree $7$, given in Entry~19(ix) of Chapter 19 of Ramanujan's second notebook \cite[p.~240]{RamanujanEarlierII}, \cite[p.~315]{BerndtIII}, \cite[Lemma~3.5]{BerndtChanZhang3}, \cite[Theorem~2.4]{Zhang}. Watson \cite{Watson3} solved this septic polynomial as well. Thus, we have
\begin{equation*}
G_{343} = 2^{1/4}7\{b_1 + b_2 + b_3 + c_1 + c_2 + c_3\}^{-1},
\end{equation*}
where
\begin{alignat*}{2}
b_1 &= (b_1^{\prime 4}\, b_2^{\prime 2}\, b_3^{\prime})^{1/7}, & \qquad c_1 &= (c_1^{\prime 4}\, c_2^{\prime 2}\, c_3^{\prime})^{1/7},\\
b_2 &= (b_2^{\prime 4}\, b_3^{\prime 2}\, b_1^{\prime})^{1/7}, & \qquad c_2 &= (c_2^{\prime 4}\, c_3^{\prime 2}\, c_1^{\prime})^{1/7},\\
b_3 &= (b_3^{\prime 4}\, b_1^{\prime 2}\, b_2^{\prime})^{1/7}, & \qquad c_3 &= (c_3^{\prime 4}\, c_1^{\prime 2}\, c_2^{\prime})^{1/7},
\end{alignat*}
and
\begin{equation*}
b'_r = -\tfrac{1}{3}(3\sigma_r + 5\tau_r),\quad  c'_r = -\tfrac{1}{3}(7 + 4\sigma_r + 2\tau_r), \qquad r = 1, 2, 3,
\end{equation*}
\begin{equation*}
\tau_1 = \sigma_3 - \sigma_2,\quad
\tau_2 = \sigma_1 - \sigma_3,\quad
\tau_3 = \sigma_3 - \sigma_1,
\end{equation*}
\begin{equation*}
\sigma_r = \frac{1}{2} + 3\cos\frac{2^r\pi}{7}, \qquad r = 1, 2, 3.
\end{equation*}

In this section, we give a closed-form expression for $G_{343}$, based on our previous results.

\begin{theorem}\label{thm:G343} We have
\begin{equation*}
G_{343} = 2^{1/4}p^{-1/7},
\end{equation*}
where
\begin{equation}\label{G343p}
p = 1 + \frac{10m^2}{s^{1/3}} - \frac{s^{1/3}}{6},
\end{equation}
and
\begin{align*}
s &= 12m^2\Big(9(7 - m^2) + \sqrt{3}\sqrt{27(m^4 + 49) + 122m^2}\Big),\\
m &= 7^{3/2}\Bigg\{1 + \bigg(\frac{\cos\frac{\pi}{7}}{2\cos^2\frac{2\pi}{7}}\bigg)^{2/7} + \bigg(\frac{\cos\frac{2\pi}{7}}{2\cos^2\frac{3\pi}{7}}\bigg)^{2/7} + \bigg(\frac{\cos\frac{3\pi}{7}}{2\cos^2\frac{\pi}{7}}\bigg)^{2/7}\Bigg\}^{-2}.
\end{align*}
\end{theorem}

\begin{proof} By taking $q = \exp(-\pi\sqrt{7})$, the expression for $m$ is obtained by Theorem~\ref{thm:enigmatic} as
\begin{equation*}
m = \frac{\varphi^2(q)}{\varphi^2(q^7)} = \frac{\varphi^2(e^{-\pi\sqrt{7}})}{\varphi^2(e^{-7\pi\sqrt{7}})}.
\end{equation*}
For Entry~\ref{entry:Ramanujan}\eqref{def:p}, by Lemma~\ref{lemma:p} with $n = 7$ and with $G_7 = 2^{1/4}$ \cite{Joubert}, \cite{Watson3}, \cite{Watson4}, \cite[p.~189]{BerndtV}, we find that
\begin{equation*}
p = \frac{2\sqrt{2}G_7}{G_{343}^7} = \frac{2\sqrt{2} \cdot 2^{1/4}}{G_{343}^7},
\end{equation*}
from which we have $G_{343} = 2^{1/4}p^{-1/7}$. On the other hand, for Entry~\ref{entry:Ramanujan}\eqref{phi-to-p}, we have
\begin{equation*}
m^4 - (2 + 5p)m^2 + (1 - p)^3 = 0.
\end{equation*}
After rearrangement, the following cubic polynomial in $p$ can be deduced:
\begin{equation*}
p^3 - 3p^2 + (3 + 5m^2)p - (m^2 - 1)^2 = 0.
\end{equation*}
We solve this equation, and choose the only real root, which is given by \eqref{G343p}.
\end{proof}

\section{Examples for Entry~\ref{entry:Ramanujan}\eqref{phi-to-p}}\label{section:examples-1}

The Borwein brothers \cite[p.~145]{BorweinBrothersPiAGM} observed first \cite{BerndtChanZhang2} that class invariants can be used to\linebreak calculate certain values of $\varphi(e^{-n\pi})$. By using ideas from Berndt's proof \cite[p.~347]{BerndtIII} of Entry~1(iii) of Chapter~20 of Ramanujan's second notebook \cite[p.~241]{RamanujanEarlierII}, \cite[p.~345]{BerndtIII}, one can deduce that\linebreak \cite[p.~330, (4.5)]{BerndtV}, \cite[(3.10)]{BerndtChan}
\begin{align}
\frac{\varphi(e^{-3\pi\sqrt{n}})}{\varphi(e^{-\pi\sqrt{n}})} &= \frac{1}{\sqrt{3}}\bigg(1 + \frac{2\sqrt{2}G_{9n}^3}{G_n^9}\bigg)^{1/4}\label{3pisqrtn}.
\intertext{Similarly, as in \cite[p.~339, (8.11)]{BerndtV} and \cite[p.~334, (5.7)]{BerndtV}, we have}
\frac{\varphi(e^{-5\pi\sqrt{n}})}{\varphi(e^{-\pi\sqrt{n}})} &= \frac{1}{\sqrt{5}}\bigg(1 + \frac{2G_{25n}}{G_n^5}\bigg)^{1/2}\label{5pisqrtn}
\intertext{and}
\frac{\varphi(e^{-9\pi\sqrt{n}})}{\varphi(e^{-\pi\sqrt{n}})} &= \frac{1}{3}\bigg(1 + \frac{\sqrt{2}G_{9n}}{G_n^3}\bigg)\label{9pisqrtn}.
\end{align}

There are two groups of values for $\varphi(q)$, which can by deduced from Entry~\ref{entry:Ramanujan}. The first one is from Entry~\ref{entry:Ramanujan}\eqref{phi-to-p}, and the second is from Entry~\ref{entry:Ramanujan}\eqref{1+u+v+w}. Now, in the spirit of Entry~\ref{entry:Ramanujan}\eqref{phi-to-p}, we give a result for $\varphi(e^{-7\pi\sqrt{n}})/\varphi(e^{-\pi\sqrt{n}})$, which is similar to those in \eqref{3pisqrtn}--\eqref{9pisqrtn}, and then we calculate the values of $\varphi(e^{-7\pi}), \varphi(e^{-7\pi\sqrt{3}}), \varphi(e^{-21\pi}),$ and $\varphi(e^{-35\pi})$.

\begin{lemma}\label{lemma:7pisqrtn} If $n$ is a positive rational number, then
\begin{equation*}
\frac{\varphi(e^{-7\pi\sqrt{n}})}{\varphi(e^{-\pi\sqrt{n}})} = \frac{1}{\sqrt{7}}\bigggg(1 + \frac{5\sqrt{2}G_{49n}}{G_n^7} + \frac{1}{2}\sqrt{\bigg(2 + \frac{10\sqrt{2}G_{49n}}{G_n^7}\bigg)^2 - 4\bigg(1 - \frac{2\sqrt{2}G_{49n}}{G_n^7}\bigg)^3}\,\bigggg)^{1/4}.
\end{equation*}
\end{lemma}
For
\begin{equation}\label{7pisqrtn-p}
p = \frac{2\sqrt{2}G_{49n}}{G_n^7},
\end{equation}
we define
\begin{equation}\label{def:mp}
m(p) := \bigg(1 + \frac{5p}{2} + \frac{1}{2}\sqrt{(2 + 5p)^2 - 4(1 - p)^3}\bigg)^{1/2}.
\end{equation}
Using $m(p)$, we can state Lemma~\ref{lemma:7pisqrtn} as
\begin{equation}\label{7pisqrtn-with-m}
\frac{\varphi(e^{-7\pi\sqrt{n}})}{\varphi(e^{-\pi\sqrt{n}})} = \frac{m^{1/2}(p)}{\sqrt{7}}.
\end{equation}
Note that $m(p)$ is a multiplier of degree $7$, defined in \eqref{def:multiplier}, with the substitution $n \mapsto (49n)^{-1}$.

\begin{proof} We apply Entry~\ref{entry:Ramanujan}\eqref{def:p},\eqref{phi-to-p}. For $q = \exp(-\pi/\sqrt{49n})$, by Lemmas~\ref{lemma:p} and \ref{lemma:G}, we find that $p = 2\sqrt{2}G_{49n}/G_n^7$. By using Lemma~\ref{lemma:phi-to-p} with Lemma~\ref{lemma:transform}, we complete the proof.
\end{proof}

The next result is from Ramanujan's first notebook \cite[p.~297]{RamanujanEarlierI}, \cite[p.~328]{BerndtV}, and proved first by Berndt and Chan \cite{BerndtChan}, \cite[pp.~336--337]{BerndtV}. Our proof uses Entry~\ref{entry:Ramanujan}\eqref{phi-to-p}, but all of these proofs depend on some of the modular equations given in Entry~19 of Chapter~19 in Ramanujan's second notebook \cite[p.~240]{RamanujanEarlierII}, \cite[pp.~314--324]{BerndtIII}. The value of $\varphi(e^{-7\pi})$, in terms of $\varphi(e^{-\pi})$ given in \eqref{e-pi}, is stated as follows.

\begin{theorem}\label{thm:e7} We have
\begin{equation*}
\frac{\varphi^2(e^{-7\pi})}{\varphi^2(e^{-\pi})} = \frac{\sqrt{13 + \sqrt{7}} + \sqrt{7 + 3\sqrt{7}}}{14}(28)^{1/8}.
\end{equation*}
\end{theorem}

\begin{proof} We apply Lemma~\ref{lemma:7pisqrtn} with $n = 1$. From \cite[p.~189]{BerndtV}, $G_1 = 1$, and from \cite{RamanujanModularPi}, \cite[p.~26]{RamanujanCollected}, \cite[p.~191]{BerndtV},
\begin{equation}\label{G49}
G_{49} = \frac{7^{1/4} + \sqrt{4 + \sqrt{7}}}{2}.
\end{equation}
Using  $4 + \sqrt{7} = (\sqrt{7} + 1)^2/2$, from \eqref{7pisqrtn-p}, we find that
\begin{equation}\label{e7pi-p}
p = \frac{2\sqrt{2}G_{49}}{G_1^7} = \sqrt{2}\bigg(7^{1/4} + \sqrt{4 + \sqrt{7}}\,\bigg) = \sqrt{7} + \sqrt{2} \cdot 7^{1/4} + 1.
\end{equation}

We make two observations. Let $a$ be a real number, and let 
\begin{equation}\label{e7pi-p_a}
p_a := \frac{(a + 1)^2 + 1}{2} = \frac{1}{2}(a^2 + 2a + 2) = \frac{a^2}{2} + a + 1.
\end{equation}
Then, straightforward algebra shows that
\begin{equation}\label{e7pi-part-i}
(2+5p_a)^2 - 4(1-p_a)^3 = \frac{1}{4}(2a^3 + 3a^2 + 10a +14)^2 + \frac{a^2}{2}(a^4 - 28),
\end{equation}
and for $a \neq 0$,
\begin{equation}\label{e7pi-part-ii}
\Bigg(\sqrt{13 + \frac{a^2}{2}} + \sqrt{7 + \frac{3a^2}{2}}\,\Bigg)^2 =  2a^2 + \sqrt{\bigg(8a + \frac{28}{a}\bigg)^2 + \frac{(3a^2 + 28)(a^4 - 28)}{a^2}} + 20.
\end{equation}

Now, set $a := (28)^{1/4} = \sqrt{2} \cdot 7^{1/4}$. Comparing \eqref{e7pi-p} with \eqref{e7pi-p_a}, we see that $p = p_a$. Furthermore, note in \eqref{e7pi-part-i} and \eqref{e7pi-part-ii} that the terms with a factor of $a^4 - 28$ vanish. Thus, from \eqref{def:mp}--\eqref{e7pi-part-ii}, we find that
\begin{align}
\frac{\varphi^2(e^{-7\pi})}{\varphi^2(e^{-\pi})} = \frac{m(p)}{7} &= \frac{1}{7}\bigg(1 + \frac{5p}{2} + \frac{1}{2}\sqrt{(2 + 5p)^2 - 4(1 - p)^3}\bigg)^{1/2}\notag\\
&= \frac{1}{7}\bigg(1 + \frac{5}{4}(a^2 + 2a + 2) + \frac{1}{4}(2a^3 + 3a^2 + 10a + 14)\bigg)^{1/2}\label{e7pi-m_a^2}\\
&= \frac{1}{7}\bigg(\frac{a}{4}\bigg(2a^2 + 8a + \frac{28}{a} + 20\bigg)\bigg)^{1/2}\notag\\
&= \frac{\sqrt{a}}{14}\Bigg(\sqrt{13 + \frac{a^2}{2}} + \sqrt{7 + \frac{3a^2}{2}}\,\Bigg)\notag\\
&= \frac{\sqrt{13 + \sqrt{7}} + \sqrt{7 + 3\sqrt{7}}}{14}(28)^{1/8}.\tag*{\qedhere}
\end{align}
\end{proof}

The next theorems appear to be new.

\begin{theorem}\label{thm:e7pisqrt3} We have
\begin{equation*}
\frac{\varphi^2(e^{-7\pi\sqrt{3}})}{\varphi^2(e^{-\pi\sqrt{3}})} = \frac{1}{42\sqrt{3}}\Big(\big(\sqrt{21}+3\big)(28)^{1/3} + 8\sqrt{3}(28)^{1/6}+4\sqrt{21} + 6\Big).
\end{equation*}
\end{theorem}

\begin{proof} We apply Lemma~\ref{lemma:7pisqrtn} with $n = 3$. From \cite{Watson3}, \cite[p.~189]{BerndtV}, $G_3 = 2^{1/12}$, and from \cite{RamanujanModularPi},\linebreak \cite[p.~28]{RamanujanCollected}, \cite{Watson5}, \cite[p.~194]{BerndtV},
\begin{equation*}
G_{147} = 2^{1/12}\bigg(\frac{1}{2} + \frac{1}{\sqrt{3}}\bigg\{\sqrt{\frac{7}{4}} - (28)^{1/6}\bigg\}\bigg)^{-1}.
\end{equation*}
Thus, from \eqref{7pisqrtn-p}, we find that
\begin{equation}\label{e7pisqrt3-p}
p = \frac{2\sqrt{2}G_{147}}{G_3^7} = 2\bigg(\frac{1}{2} + \frac{1}{\sqrt{3}}\bigg\{\sqrt{\frac{7}{4}} - (28)^{1/6}\bigg\}\bigg)^{-1}.
\end{equation}

The next observation follows by elementary algebra. Let $a$ be a real number, and let
\begin{align}
p_a &:= \frac{1}{54}(a^4 + 3a^3 + 12a^2 + 18a + 90)\label{e7pisqrt3-p_a}\\
\intertext{and}
m_a &:=  \frac{1}{6\sqrt{3}}\bigg(\frac{a}{18}(a^2 + 6)^2 + a(a + 6) + 6\bigg)\label{e7pisqrt3-m_a}.
\end{align}
Then
\begin{equation}\label{e7pisqrt3-i}
(2+5p_a)^2 - 4(1-p_a)^3 = 4\bigg(m_a^2 - 1 - \frac{5p_a}{2}\bigg)^2 - \frac{(a^6 - 756)P}{306110016},
\end{equation}
where
\begin{multline*}
P = a^{14} + 48a^{12} + 72a^{11} + 1440a^{10} + 3024a^9 + 27108a^8 + 68040a^7 + 375840a^6 \\+ 843696a^5 + 3005424a^4 + 5762016a^3 + 13576896a^2 + 15536448a + 5878656.
\end{multline*}

Now, set $a := (756)^{1/6} = 2^{1/3} \cdot \sqrt{3} \cdot 7^{1/6}$. A straightforward calculation shows that
\begin{equation*}
(a^4 + 3a^3 + 12a^2 + 18a + 90)\bigg(\frac{1}{2} + \frac{1}{\sqrt{3}}\bigg\{\sqrt{\frac{7}{4}} - (28)^{1/6}\bigg\}\bigg) = 108.
\end{equation*}
Thus, by comparing \eqref{e7pisqrt3-p} with \eqref{e7pisqrt3-p_a}, we see that $p = p_a$. Furthermore, note in \eqref{e7pisqrt3-i} that the term with a factor of $a^6 -756$ vanishes. Thus, from \eqref{def:mp}, \eqref{7pisqrtn-with-m}, and \eqref{e7pisqrt3-p}--\eqref{e7pisqrt3-i}, with some simplification, we find that
\begin{align*}
\frac{\varphi^2(e^{-7\pi\sqrt{3}})}{\varphi^2(e^{-\pi\sqrt{3}})} = \frac{m(p)}{7} = \frac{m_a}{7} &= \frac{1}{42\sqrt{3}}\bigg(\frac{a}{18}(a^2 + 6)^2 + a(a + 6) + 6\bigg)\\
&=  \frac{1}{42\sqrt{3}}\Big(\big(\sqrt{21}+3\big)(28)^{1/3} + 8\sqrt{3}(28)^{1/6}+4\sqrt{21} + 6\Big).\tag*{\qedhere}
\end{align*}
\end{proof}

By combining the value \eqref{e-sqrt3} and Theorem~\ref{thm:e7pisqrt3}, we find the evaluation of $\varphi(e^{-7\pi\sqrt{3}})$, i.e.,
\begin{equation*}
\varphi(e^{-7\pi\sqrt{3}}) = \frac{\Gamma^{3/2}\big(\frac{1}{3}\big)}{2^{7/6} 3^{5/8} \sqrt{7} \pi}\Big(\big(\sqrt{21}+3\big)(28)^{1/3} + 8\sqrt{3}(28)^{1/6}+4\sqrt{21} + 6\Big)^{1/2}.
\end{equation*}
The next values are derived in terms of $\varphi(e^{-\pi})$ given in \eqref{e-pi}.

\begin{theorem}\label{thm:e21} We have
\begin{equation*}
\frac{\varphi(e^{-21\pi})}{\varphi(e^{-\pi})} = \bigg(\frac{m(p)}{7(6\sqrt{3} - 9)^{1/2}}\bigg)^{1/2},
\end{equation*}
where
\begin{equation}\label{e21-p}
p = \sqrt{2}\big(2-\sqrt{3}\big)\sqrt{\sqrt{3} + \sqrt{7}}\sqrt{2 + \sqrt{7} + \sqrt{7 + 4\sqrt{7}}}\sqrt{\frac{\sqrt{3 + \sqrt{7}} + (6\sqrt{7})^{1/4}}{\sqrt{3 + \sqrt{7}} - (6\sqrt{7})^{1/4}}}
\end{equation}
and $m(p)$ is given in \eqref{def:mp}.
\end{theorem}

\begin{proof} We apply Lemma~\ref{lemma:7pisqrtn} with $n = 9$. From \cite{RamanujanModularPi}, \cite[p.~24]{RamanujanCollected}, \cite[p.~189]{BerndtV},
\begin{equation*}
G_9 = \Bigg(\frac{1 + \sqrt{3}}{\sqrt{2}}\Bigg)^{1/3},
\end{equation*}
and from \cite{RamanujanModularPi}, \cite[p.~29]{RamanujanCollected}, \cite{BerndtChanZhang}, \cite[p.~197]{BerndtV},
\begin{equation}\label{G441}
G_{441} = \sqrt{\frac{\sqrt{3} + \sqrt{7}}{2}}\big(2 + \sqrt{3}\big)^{1/6}\sqrt{\frac{2 + \sqrt{7} + \sqrt{7 + 4\sqrt{7}}}{2}}\sqrt{\frac{\sqrt{3 + \sqrt{7}} + (6\sqrt{7})^{1/4}}{\sqrt{3 + \sqrt{7}} - (6\sqrt{7})^{1/4}}}.
\end{equation}
An elementary calculation shows that
\begin{equation*}
\frac{2^{5/3}(2+\sqrt{3})^{1/6}}{(1+\sqrt{3})^{7/3}} = \sqrt{2}\big(2 - \sqrt{3}\big).
\end{equation*}
From \eqref{7pisqrtn-p}, $p = 2\sqrt{2}G_{441}/G_9^7$. Using the values for $G_9$ and $G_{441}$ given above, we deduce \eqref{e21-p}.

From \cite[p.~284]{RamanujanEarlierI}, \cite{BerndtChan}, \cite[pp.~327, 329--331]{BerndtV},
\begin{equation}\label{e3}
\frac{\varphi(e^{-3\pi})}{\varphi(e^{-\pi})} = \frac{1}{(6\sqrt{3} - 9)^{1/4}}.
\end{equation}
Combining Lemma~\ref{lemma:7pisqrtn} and \eqref{e3} completes the proof.
\end{proof}

Another expression for $\varphi(e^{-21\pi})$ can be obtained by using \eqref{3pisqrtn} with $n = 49$ and with the values $G_{441}$ from \eqref{G441}, and $G_{49}$ from \eqref{G49}. Combining the result with the value of $\varphi(e^{-7\pi})$ from Theorem~\ref{thm:e7}, after some simplification we find that
\begin{multline*}
\frac{\varphi(e^{-21\pi})}{\varphi(e^{-\pi})} = \biggg(\frac{\sqrt{13 + \sqrt{7}} + \sqrt{7 + 3\sqrt{7}}}{42}\biggg)^{1/2}(28)^{1/16}\\
\times \biggg\{1 + \frac{1}{4}\sqrt{2}\sqrt{2 + \sqrt{3}}\biggg(\big(\sqrt{3} + \sqrt{7}\big)\bigg(2 + \sqrt{7} + \sqrt{7 + 4\sqrt{7}}\bigg)\\
\times \bigg(22 + 8\sqrt{7} - \frac{1}{2}\big(19 + 7\sqrt{7}\big)\sqrt{2\sqrt{7}}\bigg)\frac{\sqrt{3 + \sqrt{7}} + \big(6\sqrt{7}\big)^{1/4}}{\sqrt{3 + \sqrt{7}} - \big(6\sqrt{7}\big)^{1/4}}\biggg)^{3/2}\biggg\}^{1/4}.
\end{multline*}

\begin{theorem}\label{thm:e35} We have
\begin{equation*}
\frac{\varphi(e^{-35\pi})}{\varphi(e^{-\pi})} = \bigg(\frac{m(p)}{35(\sqrt{5} - 2)}\bigg)^{1/2},
\end{equation*}
where
\begin{multline}\label{e35-p}
p = \frac{1}{4}\big(9 - 4\sqrt{5}\big)\sqrt{\sqrt{14} + \sqrt{10}}\bigg(7^{1/4}+\sqrt{4 + \sqrt{7}}\,\bigg)^{3/2}\\
\times \Bigg(\sqrt{43 + 15\sqrt{7} + (8 + 3\sqrt{7})\sqrt{10\sqrt{7}}} + \sqrt{35 + 15\sqrt{7} + (8 + 3\sqrt{7})\sqrt{10\sqrt{7}}}\,\Bigg)
\end{multline}
and $m(p)$ is given in \eqref{def:mp}.
\end{theorem}

\begin{proof} We apply Lemma~\ref{lemma:7pisqrtn} with $n = 25$. From \cite{RamanujanModularPi}, \cite[p.~26]{RamanujanCollected}, \cite[p.~190]{BerndtV},
\begin{equation*}
G_{25} = \frac{1 + \sqrt{5}}{2},
\end{equation*}
and from \cite{RamanujanModularPi}, \cite[p.~30]{RamanujanCollected}, \cite{WatsonII}, \cite[p.~199]{BerndtV},
\begin{multline}\label{G1225}
G_{1225} = \frac{1 + \sqrt{5}}{2}\big(6 + \sqrt{35}\big)^{1/4}\biggg(\frac{7^{1/4} + \sqrt{4 + \sqrt{7}}}{2}\biggg)^{3/2}\\
\times \Biggggg(\sqrt{\frac{43 + 15\sqrt{7} + (8 + 3\sqrt{7})\sqrt{10\sqrt{7}}}{8}} + \sqrt{\frac{35 + 15\sqrt{7} + (8 + 3\sqrt{7})\sqrt{10\sqrt{7}}}{8}}\,\Biggggg).
\end{multline}
A simple calculation shows that
\begin{equation*}
\frac{1}{G_{25}^6} = \bigg(\frac{2}{1 + \sqrt{5}}\bigg)^6 = 9 - 4\sqrt{5}.
\end{equation*}
We use
\begin{equation*}
\big(6+\sqrt{35}\big)^{1/4} = \sqrt{\frac{\sqrt{14} + \sqrt{10}}{2}},
\end{equation*}
as it is given by Watson \cite{WatsonII}, which also can be verified directly. From \eqref{7pisqrtn-p}, $p = 2\sqrt{2}G_{1225}/G_{25}^7$. Using the values for $G_{25}$ and $G_{1225}$ given above, we deduce \eqref{e35-p}.

From \cite[p.~285]{RamanujanEarlierI}, \cite{BerndtChan}, \cite[pp.~327--329]{BerndtV},
\begin{equation}\label{e5}
\frac{\varphi(e^{-5\pi})}{\varphi(e^{-\pi})} = \frac{1}{(5\sqrt{5} - 10)^{1/2}}.
\end{equation}
Combining Lemma~\ref{lemma:7pisqrtn} and \eqref{e5} completes the proof.
\end{proof}

Another expression for $\varphi(e^{-35\pi})$ can be obtained by using \eqref{5pisqrtn} with $n = 49$ and with the values $G_{1225}$ from \eqref{G1225}, and $G_{49}$ from \eqref{G49}. Combining the result with the value of $\varphi(e^{-7\pi})$ from Theorem~\ref{thm:e7}, after some simplification we find that
\begin{multline*}
\frac{\varphi(e^{-35\pi})}{\varphi(e^{-\pi})} = \biggg(\frac{\sqrt{13 + \sqrt{7}} + \sqrt{7 + 3\sqrt{7}}}{70}\biggg)^{1/2}(28)^{1/16}\\
\times \Bigg\{1 + \frac{1}{4}\big(1 + \sqrt{5}\big)\sqrt{\sqrt{7} + \sqrt{5}}\bigg(16466 + 6223\sqrt{7} - \frac{7}{2}\big(2045 + 773\sqrt{7}\big)\sqrt{2\sqrt{7}}\bigg)^{1/4}\\
\times \Bigg(\sqrt{43 + 15\sqrt{7} + \big(8 + 3\sqrt{7}\big)\sqrt{10\sqrt{7}}} + \sqrt{35 + 15\sqrt{7} + \big(8 + 3\sqrt{7}\big)\sqrt{10\sqrt{7}}}\Bigg)\Bigg\}^{1/2}.
\end{multline*}

\section{Examples for Entry~\ref{entry:Ramanujan}\eqref{1+u+v+w}}\label{section:examples-2}

Now, we review our proof for Theorem~\ref{thm:enigmatic}, which is in the form of Entry~\ref{entry:Ramanujan}\eqref{1+u+v+w}. Then, we give the value of $\varphi(e^{-49\pi})$, as a second illustration of Entry~\ref{entry:Ramanujan}\eqref{1+u+v+w}. We use the results from the proof of Theorem~\ref{thm:e7}.

Let $a \in \{0, (28)^{1/4}\}$. After \eqref{e7pi-p_a}, let
\begin{equation}\label{e49-p_a}
p_a := \frac{(a + 1)^2 + 1}{2} = \frac{1}{2}(a^2 + 2a + 2).
\end{equation}
Since the second term on the right-hand side of \eqref{e7pi-part-i} vanishes for $a \in \{0, (28)^{1/4}\}$, after \eqref{e7pi-m_a^2}, let
\begin{equation}\label{e49-m_a}
m_a := \bigg(\frac{1}{2}(a^3 + 4a^2 + 10a + 14)\bigg)^{1/2}.
\end{equation}
Furthermore, let
\begin{multline}
r_a(\xi) := \xi^3 + 2\xi^2(1 + 3p_a - m_a^2) + \xi p_a^2(p_a+4) - p_a^4\\
 = \xi^3 - (a^3 + a^2 + 4a + 6)\xi^2 + \frac{1}{8}(a^2 + 2a + 2)^2(a^2 + 2a + 10)\xi - \frac{1}{16}(a^2 + 2a + 2)^4\label{e49-r_a}.
\end{multline}

\begin{proof}[Comments on the proof of Theorem~\ref{thm:enigmatic}] Set $a := 0$. Using the notations of the proof of Theorem~\ref{thm:enigmatic}, from \eqref{e49-p_a}, \eqref{e49-m_a}, and \eqref{e49-r_a}, we find that $p = p_a = 1$, $\varphi^4(q)/\varphi^4(q^7) = m_a^2 = 7$, and
\begin{equation*}
r(\xi) = r_a(\xi) = \xi^3 - 6\xi^2 + 5\xi - 1.
\end{equation*}

By Lemma~\ref{lemma:roots}, we know that $\cos(\pi/7)$ is a root of $U_6$. Thus, by using the power-reduction formula \cite[p.~32]{GradshteynRyzhik7}
\begin{equation*}
\cos^2\bigg(\frac{\pi}{7}\bigg) = \frac{1}{2}\bigg(\cos\bigg(\frac{2\pi}{7}\bigg) + 1\bigg),
\end{equation*}
we can factor $r$ over $\mathbb{Q}(\cos(\pi/7))$ as
\begin{equation}\label{e7sqrt7-factor}
r(\xi) = (\xi - \alpha)(\xi - \beta)(\xi - \gamma),
\end{equation}
where
\begin{align*}
\alpha = 2 + 2\cos\bigg(\frac{\pi}{7}\bigg) + 2\cos\bigg(\frac{2\pi}{7}\bigg),\\
\beta = 3 - 4\cos\bigg(\frac{\pi}{7}\bigg) + 2\cos\bigg(\frac{2\pi}{7}\bigg),\\
\gamma = 1 + 2\cos\bigg(\frac{\pi}{7}\bigg) - 4\cos\bigg(\frac{2\pi}{7}\bigg).
\end{align*}

In the same manner as in the proof of Lemma~\ref{lemma:trig}, we deduce
\begin{equation*}
(\alpha , \beta, \gamma) = \Bigg(\frac{1}{(2\cos\frac{3\pi}{7})^2}, \frac{1}{(2\cos\frac{2\pi}{7})^2}, \frac{1}{(2\cos\frac{\pi}{7})^2}\Bigg),
\end{equation*}
where the order of the roots is determined by Lemmas~\ref{lemma:a,b,c-order} and \ref{lemma:trig}. We construct $u, v,$ and $w$, and the proof is complete.
\end{proof}

After these preliminaries, we derive the value of $\varphi(e^{-49\pi})$. The corresponding polynomial $r$, and its roots are more complicated than in Ramanujan's example in Entry~\ref{entry:Ramanujan}\eqref{enigmatic}. We used \emph{Mathematica} for polynomial factorization and numerical evaluations.

\begin{theorem}\label{thm:e49} We have
\begin{align}\label{e49-1+u+v+w}
\frac{\varphi(e^{-49\pi})}{\varphi(e^{-\pi})} = \frac{1}{7}(1 + u + v + w),
\end{align}
where
\begin{equation*}
u = \bigg(\frac{\alpha^2 p}{\beta}\bigg)^{1/7}, \quad
v = \bigg(\frac{\beta^2 p}{\gamma}\bigg)^{1/7}, \quad
w = \bigg(\frac{\gamma^2 p}{\alpha}\bigg)^{1/7},
\end{equation*}
and
\begin{align*}
p &= \sqrt{7} + \sqrt{2}\cdot 7^{1/4} + 1,\\
\alpha &= \frac{1}{\sqrt{7}}\bigg\{\frac{2}{3}\big(\sqrt{7} + 2\big)\big(5 + 3\sqrt{2}\cdot7^{1/4} - \sqrt{7}\big)
+ 2\big(\sqrt{7} - 1\big)\big(1 - \sqrt{2}\cdot7^{1/4} + \sqrt{7}\big)\cos\bigg(\frac{\pi}{7}\bigg)\\
&\qquad\qquad + \frac{2}{3}\big(\sqrt{7} - 1\big)\big(3\sqrt{2}\cdot 7^{1/4} - \sqrt{7} - 1\big)\cos\bigg(\frac{2\pi}{7}\bigg)\bigg\},\\
\beta &= \frac{1}{\sqrt{7}}\bigg\{\frac{1}{9}\big(\sqrt{7} + 5\big)\big(13 + 9\sqrt{2}\cdot 7^{1/4} + \sqrt{7}\big)\\
&\qquad\qquad - 8\cos\bigg(\frac{\pi}{7}\bigg) + 2\big(\sqrt{7} - 1\big)\big(1 - \sqrt{2}\cdot 7^{1/4} + \sqrt{7}\big)\cos\bigg(\frac{2\pi}{7}\bigg)\bigg\},\\
\gamma &= \frac{1}{\sqrt{7}}\bigg\{\frac{1}{3}\big(\sqrt{7} + 5\big)\big(1 + 3\sqrt{2}\cdot 7^{1/4} + \sqrt{7}\big)\\
&\qquad\qquad + \frac{2}{3}\big(\sqrt{7} - 1\big)\big(3\sqrt{2}\cdot 7^{1/4} - \sqrt{7} - 1\big)\cos\bigg(\frac{\pi}{7}\bigg) - 8\cos\bigg(\frac{2\pi}{7}\bigg)\bigg\}.
\end{align*}
\end{theorem}

\begin{proof} We use the results in Entry~\ref{entry:Ramanujan}\eqref{1+u+v+w}--\eqref{r} with $q = \exp(-\pi/7)$. First, by using Lemma~\ref{lemma:transform}, we rewrite Entry~\ref{entry:Ramanujan}\eqref{1+u+v+w} as in \eqref{e49-1+u+v+w}.

Then, set $a := (28)^{1/4} = \sqrt{2} \cdot 7^{1/4}$, and consider $p_a, m_a,$ and $r_a$ from \eqref{e49-p_a}, \eqref{e49-m_a}, and \eqref{e49-r_a}. By Lemma~\ref{lemma:G}, and by Lemma~\ref{lemma:p} with $n = 1/49$, for Entry~\ref{entry:Ramanujan}\eqref{def:p}, we have $p = 2\sqrt{2}G_{49}/G_1^7$. Thus, comparing \eqref{e7pi-p} with \eqref{e49-p_a}, we see that $p = p_a = (a^2/2) + a + 1$. For Lemma~\ref{entry:Ramanujan}\eqref{phi-to-p}, by using Lemmas~\ref{lemma:phi-to-p} and \ref{lemma:transform}, we arrive at
\begin{equation*}
\frac{\varphi^4(q)}{\varphi^4(q^7)} = \frac{\varphi^4(e^{-\pi/7})}{\varphi^4(e^{-\pi})} = 49\frac{\varphi^4(e^{-7\pi})}{\varphi^4(e^{-\pi})} = m_a^2 = \frac{1}{2}(a^3 + 4a^2 + 10a + 14),
\end{equation*}
where the last equation is obtained by the comparison of \eqref{e7pi-m_a^2} and \eqref{e49-m_a}. Thus, by comparing Entry~\ref{entry:Ramanujan}\eqref{r} with \eqref{e49-r_a}, we find that $r(\xi) = r_a(\xi)$, where we remind readers that
\begin{equation*}
r_a(\xi) = \xi^3 - (a^3 + a^2 + 4a + 6)\xi^2 + \frac{1}{8}(a^2 + 2a + 2)^2(a^2 + 2a + 10)\xi - \frac{1}{16}(a^2 + 2a + 2)^4.
\end{equation*}

Now, because of Lemma~\ref{lemma:roots}, following \eqref{e7sqrt7-factor}, we factor $r_a$ over $\mathbb{Q}(\cos(\pi/7), a)$ as
\begin{align*}
r_a(\xi) = (\xi - \alpha)(\xi - \beta)(\xi - \gamma),
\end{align*}
where
\begin{align*}
\alpha &= \frac{2}{a^2}\bigg\{a^3 + a^2 + 4a + 2 + (2a - a^3 + 12)\cos\bigg(\frac{\pi}{7}\bigg) + (a^3 - 2a - 4)\cos\bigg(\frac{2\pi}{7}\bigg)\bigg\},\\
\beta &= \frac{2}{a^2}\bigg\{\frac{a^3}{2} + a^2 + 5a + 8 - 8\cos\bigg(\frac{\pi}{7}\bigg) + (2a - a^3 + 12)\cos\bigg(\frac{2\pi}{7}\bigg)\bigg\},\\
\gamma &=\frac{2}{a^2}\bigg\{\frac{a^3}{2} + a^2 + 5a + 4 + (a^3 - 2a - 4)\cos\bigg(\frac{\pi}{7}\bigg) - 8\cos\bigg(\frac{2\pi}{7}\bigg)\bigg\}.
\end{align*}

Numerical evaluations exclude all possible orders of $\alpha,\beta,$ and $\gamma$, which do not meet the conditions in Lemma~\ref{lemma:a,b,c-order}\eqref{a,b,c-order-cond1},\eqref{a,b,c-order-cond2}. Thus, we find that $(\alpha, \beta, \gamma)$ is the correct order of the roots, and by Entry~\ref{entry:Ramanujan}\eqref{u,v,w}, the proof is complete.
\end{proof}

From Theorem~\ref{thm:e49}, similarly as we have seen in Theorem~\ref{thm:G343}, the value of $G_{2401}$ can be determined. As from Theorem~\ref{thm:e7} in Theorem~\ref{thm:e49}, by using Theorems~\ref{thm:e7pisqrt3}, \ref{thm:e21}, and~\ref{thm:e35}, analogous results can be obtained for $\varphi(e^{-49\pi\sqrt{3}}), \varphi(e^{-147\pi}),$ and $\varphi(e^{-245\pi})$, respectively. These values can be expressed by using the solutions of the corresponding cubic polynomials, but their structure seems much more complicated.

Based on a remark at the end of Section 12 of Chapter 20 of Ramanujan's second notebook \cite[p.~247]{RamanujanEarlierII}, \cite[p.~400]{BerndtIII}, we believe that the septic identity in Entry~\ref{entry:Ramanujan} is a special case of a much more general result. We will continue our investigation in this direction with the description of the analogous cubic and quintic identities.

\subsubsection*{Acknowledgments} I started to work on this paper as an independent scholar in Győr, Hungary, and finalized it at University of Tromsø -- The Arctic University of Norway, in Tromsø, Norway.\linebreak The comments of Professors Bruce C. Berndt and Trygve Johnsen are highly appreciated.

\medskip


\medskip

\begin{thebibliography}{10}

\bibitem{Abel}
N.~H. Abel.
\newblock {Recherches sur les fonctions elliptiques}.
\newblock {\em J. Reine Angew. Math.}, 3:160--190, 1828.

\bibitem{AndrewsBerndtII}
G.~E. Andrews and B.~C. Berndt.
\newblock {\em {Ramanujan's Lost Notebook, Part II.}}
\newblock Springer, New York, 2009.

\bibitem{BankoffGarfunkel}
L.~Bankoff and J.~Garfunkel.
\newblock {The heptagonal triangle}.
\newblock {\em Math. Mag.}, 46:7--19, 1973.

\bibitem{BerndtIII}
B.~C. Berndt.
\newblock {\em {Ramanujan's Notebooks, Part III.}}
\newblock Springer-Verlag, New York, 1991.

\bibitem{BerndtClassicalToModern}
B.~C. Berndt.
\newblock {Ramanujan's theory of theta-function}.
\newblock In {\em {Theta Functions, From the Classical to the Modern,
  \emph{M.~Ram~Murty, ed., Centre de Recherches Math\'{e}matiques Proceedings
  and Lecture Notes, Vol.~1, pp.~1--63}}}, American Mathematical Society,
  Providence, RI, 1993.

\bibitem{BerndtIV}
B.~C. Berndt.
\newblock {\em {Ramanujan's Notebooks, Part IV.}}
\newblock Springer-Verlag, New York, 1994.

\bibitem{BerndtV}
B.~C. Berndt.
\newblock {\em {Ramanujan's Notebooks, Part V.}}
\newblock Springer-Verlag, New York, 1998.

\bibitem{BerndtForty}
B.~C. Berndt.
\newblock {The remaining 40\% of Ramanujan’s lost notebook}.
\newblock In {\em Number Theory and its Applications,
  \emph{Surikaisekikenkyuusho Kokyuuroku, No.~1060, pp.~111--118}}, RIMS Kyoto
  University, Kyoto, 1998.

\bibitem{BerndtChan}
B.~C. Berndt and H.~H. Chan.
\newblock {Ramanujan's explicit values for the classical theta function}.
\newblock {\em Mathematika}, 42(2):278--294, 1995.

\bibitem{BerndtChanZhang}
B.~C. Berndt, H.~H. Chan, and L.-C. Zhang.
\newblock {Ramanujan's class invariants and cubic continued fraction}.
\newblock {\em Acta Arith.}, 73(1):67--85, 1995.

\bibitem{BerndtChanZhang3}
B.~C. Berndt, H.~H. Chan, and L.-C. Zhang.
\newblock {Ramanujan's class invariants, Kronecker's limit formula, and modular
  equations}.
\newblock {\em Trans. Amer. Math. Soc.}, 349(6):2125--2173, 1997.

\bibitem{BerndtChanZhang2}
B.~C. Berndt, H.~H. Chan, and L.-C. Zhang.
\newblock {Ramanujan's class invariants with applications to the values of
  $q$-continued fractions and theta functions}.
\newblock In {\em {Special Functions, $q$-Series and Related Topics}},
  {M.~E.~H. Ismail, D.~R. Masson, and M. Rahman, eds., Fields Institute
  Communications Series, Vol.~14, pp.~37--53}. {American Mathematical Society},
  {Providence, RI}, 1997.

\bibitem{BorweinBrothersPiAGM}
J.~M. Borwein and P.~B. Borwein.
\newblock {\em {Pi and the AGM}}.
\newblock Wiley, New York, 1987.

\bibitem{BorweinZucker}
J.~M. Borwein and I.~J. Zucker.
\newblock {Fast evaluation of the gamma function for small rational fractions
  using complete elliptic integrals of the first kind}.
\newblock {\em IMA J. Numer. Anal.}, 12(4):519--526, 1992.

\bibitem{Cox}
D.~A. Cox.
\newblock {\em {Primes of the Form $x^2 + ny^2$}}.
\newblock Wiley, New York, 1989.

\bibitem{GradshteynRyzhik7}
I.~S. Gradshteyn and I.~M. Ryzhik.
\newblock {\em {Table of Integrals, Series, and Products}}.
\newblock 7th ed., A.~Jeffrey and D.~Zwillinger, eds., Academic Press, New
  York, 2007.

\bibitem{Jacobi}
C.~G.~J. Jacobi.
\newblock {\em {Fundamenta nova theoriae functionum ellipticarum}}.
\newblock Sumptibus fratrum Borntr\ae ger, {Regiomonti}, 1829.

\bibitem{Joubert}
P.~Joubert.
\newblock {Sur la th\'{e}orie des fonctions elliptiques et son application
  \`{a} la th\'{e}orie des nombres}.
\newblock {\em Comptes rendus}, 50:907--912, 1860.

\bibitem{JoyceZucker}
G.~S. Joyce and I.~J. Zucker.
\newblock {Special values of the hypergeometric series}.
\newblock {\em Math. Proc. Cambridge Philos. Soc.}, 109(2):257--261, 1991.

\bibitem{MasonHandscomb}
J.~C. Mason and D.~C. Handscomb.
\newblock {\em {Chebyshev Polynomials}}.
\newblock CRC Press, Boca Raton, 2003.

\bibitem{RamanujanModularPi}
S.~Ramanujan.
\newblock {Modular equations and approximations to {$\pi$}}.
\newblock {\em Quart. J. Math.}, 45:350--372, 1914.

\bibitem{RamanujanCollected}
S.~Ramanujan.
\newblock {\em {Collected Papers of Srinivasa Ramanujan}}.
\newblock G.~H. Hardy, P.~V. Seshu Aiyar, and B.~M. Wilson, eds., Cambridge
  University Press, Cambridge, 1927.

\bibitem{RamanujanEarlierI}
S.~Ramanujan.
\newblock {\em {Notebooks of Srinivasa Ramanujan, Volume I.}}
\newblock Tata Institute of Fundamental Research, Bombay, 1957.

\bibitem{RamanujanEarlierII}
S.~Ramanujan.
\newblock {\em {Notebooks of Srinivasa Ramanujan, Volume II.}}
\newblock Tata Institute of Fundamental Research, Bombay, 1957.

\bibitem{RamanujanLost}
S.~Ramanujan.
\newblock {\em {The Lost Notebook and Other Unpublished Papers}}.
\newblock Narosa, New Delhi, 1988.

\bibitem{SelbergChowla}
A.~Selberg and S.~Chowla.
\newblock {On Epstein's zeta function}.
\newblock {\em J. Reine Angew. Math.}, 227:86--110, 1967.

\bibitem{Son}
S.~H. Son.
\newblock {Septic theta function identities in Ramanujan's lost notebook}.
\newblock {\em Acta Arith.}, 98(4):361--374, 2001.

\bibitem{Son2}
S.~H. Son.
\newblock {Ramanujan's symmetric theta functions in his lost notebook}.
\newblock In {\em {Special Functions and Orthogo{\-}nal Polynomials}},
  {D.~Dominici and R.~S.~Maier, eds., Contemp. Math., Vol. 471, pp.~187--202}.
  {American Mathematical Society}, {Providence, RI}, 2008.

\bibitem{WatsonII}
G.~N. Watson.
\newblock {Some singular moduli (II)}.
\newblock {\em Quart. J. Math.}, 3(1):189--212, 1932.

\bibitem{Watson4}
G.~N. Watson.
\newblock {Singular moduli (4)}.
\newblock {\em Acta Arith.}, 1(2):284--323, 1935.

\bibitem{Watson3}
G.~N. Watson.
\newblock {Singular moduli (3)}.
\newblock {\em Proc. London Math. Soc.}, 40(1):83--142, 1936.

\bibitem{Watson5}
G.~N. Watson.
\newblock {Singular moduli (5)}.
\newblock {\em Proc. London Math. Soc.}, 42(1):377--397, 1937.

\bibitem{Weber}
H.~Weber.
\newblock {\em {Lehrbuch der Algebra, Dritter Band}}.
\newblock {2nd ed., Druck und Verlag von Friedrich Vieweg und Sohn},
  Braunschweig, 1908.

\bibitem{WhittakerWatson}
E.~T. Whittaker and G.~N. Watson.
\newblock {\em {A Course of Modern Analysis}}.
\newblock 4th ed., Cambridge University Press, Cambridge, 1950.

\bibitem{YiThesis}
J.~Yi.
\newblock {\em {The Construction and Applications of Modular Equations}}.
\newblock PhD thesis, {University of Illinois at Urbana--Champaign}, Urbana,
  Illinois, 2001.

\bibitem{Zhang}
L.-C. Zhang.
\newblock {Ramanujan's class invariants, Kronecker's limit formula and modular
  equations (III)}.
\newblock {\em Acta Arith.}, 82(4):379--392, 1997.

\bibitem{Zucker}
I.~J. Zucker.
\newblock {The evaluation in terms of $\Gamma$-functions of the periods of
  elliptic curves admitting complex multiplication}.
\newblock {\em Math. Proc. Cambridge Philos. Soc.}, 82(1):111--118, 1977.

\bibitem{ZuckerJoyceII}
I.~J. Zucker and G.~S. Joyce.
\newblock {Special values of the hypergeometric series II}.
\newblock {\em Math. Proc. Cambridge Philos. Soc.}, 131(2):309--319, 2001.

\end{thebibliography}
\end{document}